\documentclass[11pt,reqno]{amsart}

\usepackage[T1]{fontenc}
\pdfoutput=1
\usepackage{mathptmx,microtype,hyperref,graphicx,dsfont,booktabs,mathtools}
\usepackage{amsthm,amsmath,amssymb}
\usepackage{enumitem}
\usepackage[nameinlink,capitalise]{cleveref}
\usepackage[font=footnotesize,margin=2cm]{caption}
\usepackage{cite}
\usepackage{todonotes}
\newcommand{\Mod}[1]{\ (\mathrm{mod}\ #1)}
\crefname{theorem}{Theorem}{Theorems}
\crefname{thm}{Theorem}{Theorems}
\crefname{mainthm}{Theorem}{Theorems}
\crefname{lemma}{Lemma}{Lemmas}
\crefname{lem}{Lemma}{Lemmas}
\crefname{remark}{Remark}{Remarks}
\crefname{claim}{Claim}{Claims}
\crefname{subclaim}{Sub-claim}{Sub-claims}
\crefname{prop}{Proposition}{Propositions}
\crefname{proposition}{Proposition}{Propositions}
\crefname{defn}{Definition}{Definitions}
\crefname{corollary}{Corollary}{Corollaries}
\crefname{conjecture}{Conjecture}{Conjectures}
\crefname{question}{Question}{Questions}
\crefname{chapter}{Chapter}{Chapters}
\crefname{section}{Section}{Sections}
\crefname{figure}{Figure}{Figures}
\newcommand{\crefrangeconjunction}{--}

\theoremstyle{plain}

\newtheorem{thm}{Theorem}
\newtheorem*{thm*}{Theorem}
\newtheorem{lemma}[thm]{Lemma}

\theoremstyle{definition}
\newtheorem{defn}[thm]{Definition}

\theoremstyle{remark}
\newtheorem*{remark}{Remark}

\newcommand{\eps}{\varepsilon}

\newcommand{\R}{{\mathbb R}}
\newcommand{\Z}{{\mathbb Z}}

\newcommand{\cS}{{\mathcal S}}
\newcommand{\cG}{{\mathcal G}}

\newcommand{\cX}{{\mathcal X}}
\newcommand{\cT}{{\mathcal T}}

\newcommand{\cK}{{\mathcal K}}

\newcommand{\cY}{{\mathcal Y}}

\newcommand{\sgn}{\textrm{sgn}}
\newcommand{\tour}{{\rm Tour}}
\newcommand{\score}{{\rm Score}}
\newcommand{\win}{{\rm Win}}
\newcommand{\ig}{{\rm IntGr}}
\renewcommand{\le}{\leqslant}
\renewcommand{\ge}{\geqslant}

\renewcommand\ell{l}

\usepackage{todonotes}

\author[B. Kolesnik]{Brett Kolesnik}
\address{Department of Statistics, University of Warwick}
\email{brett.kolesnik@warwick.ac.uk}

\author[R. Mitchell]{Rivka Mitchell}
\address{Department of Mathematics, University of Oxford}
\email{rivka.mitchell@maths.ox.ac.uk}

\author[T. Przyby{\l}owski]{Tomasz Przyby{\l}owski}
\address{Department of Mathematics, University of Oxford}
\email{przybylowski@maths.ox.ac.uk}


\begin{document}

\title
[Coxeter interchange graphs]
{Coxeter interchange graphs}

\begin{abstract}

Brualdi and Li introduced 
tournament interchange graphs. 
In such a graph, each vertex represents a tournament. 
Traversing an edge corresponds to reversing 
a cyclically directed triangle. 
Such a triangle is neutral, in 
that its reversal does not 
affect the score sequence. 
An interchange graph encodes the combinatorics
of the set of tournaments with a given score sequence, 
or equivalently, of a given fiber of the classical permutahedron
from discrete geometry. 

Coxeter tournaments were introduced by the 
first author and Sanchez, in relation to the 
Coxeter permutahedra in Ardila, Castillo, Eur and Postnikov.
Coxeter tournaments have collaborative and solitaire games, 
in addition to the usual competitive games
in classical tournaments. 

We introduce 
Coxeter interchange graphs. 
These graphs are more intricate, 
as there are multiple neutral structures
at play, which interact with one another. 
Our main result shows that the
Coxeter interchange graphs are regular,
and we describe the degree 
geometrically, 
in terms of 
distances in the 
Coxeter permutahedra. 
We also characterize the set of  
score sequences of Coxeter tournaments, 
generalizing a classical result of Landau.

\end{abstract}

\maketitle

\begin{figure}[h!]
\centering
\includegraphics[scale=1]{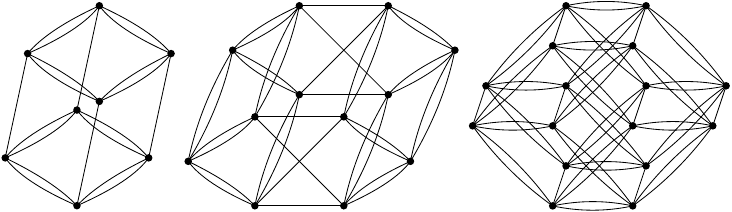}
\caption{
For a root system of type $\Phi$, 
the permutahedron $\Pi(\Phi)$ 
is the convex hull of the orbit of the Weyl vector 
${\bf s}(\Phi)$ under the Weyl group. 
In an interchange graph,  
vertices represent tournaments with 
score sequence ${\bf s}$,   
and edges correspond to the reorientation of 
one of the small 
neutral structures in \cref{F_genD,F_genB,F_genC}. 
\cref{T_MainCount} 
shows that
the degree is $(\| {\bf s}(\Phi)\|^2-\| {\bf s}\|^2)/2$. 
The graphs above arise in type $C_3$
for ${\bf s}=(2,0,0)$, $(-1,0,1)$ and
$(0,0,0)$, where 
${\bf s}(C_3)=(1,2,3)$. 
The degrees (5, 6 and 7)
increase
as ${\bf s}$ approaches the center $(0,0,0)$
of $\Pi(C_3)$ in \cref{F_permC}. 
}
\label{F_IGs}
\end{figure}

\newpage

\section{Background}
\label{S_back}

\subsection{Tournaments}
\label{S_tour}

A {\it tournament} is an orientation of the complete graph $K_n$. 
We encode a tournament 
$T=\{w_{ij}:i>j\}$ using values $w_{ij}\in\{0,1\}$. 
If $w_{ij}=1$ we orient the edge $\{i,j\}$ 
as $i\to j$, 
and otherwise $i\leftarrow j$ if $w_{ij}=0$. 
We think of vertices as players  
and edges as competitive games, 
directed away from the winner. 
The  {\it win sequence} 
\[
{\bf w}(T)=\sum_{i>j}[w_{ij}{\bf e}_i+(1-w_{ij}){\bf e}_j]
\]
lists the total number of wins by each player, 
where ${\bf e}_i\in\Z^n$ are the usual basis vectors. 

The {\it standard win sequence,} corresponding to the {\it transitive} 
(acyclic) tournament, 
in which all $w_{ij}=1$ 
(that is, player $i$ wins against all players $j<i$) is denoted by  
\[
{\bf w}_n= (0,1,\ldots,n-1). 
\]

The permutahedron $\Pi_{n-1}$ is 
a classical object in discrete geometry, obtained as 
the convex hull of 
${\bf w}_n$ and its permutations. 
Connections between tournaments
and the geometry of $\Pi_{n-1}$ are well known; 
see Stanley \cite{S80}
(and cf.\ \cite{KS24}). 
Notably, by Rado \cite{Rad52} 
and Landau \cite{Lan53}, 
the set $\win(n)$
of all win sequences
is the set 
$\Z^n\cap \Pi_{n-1}$
of all lattice points
in $\Pi_{n-1}$. 

For ${\bf x}, {\bf y}\in\R^n$, we say that ${\bf x}$ {\it majorizes} ${\bf y}$, and 
write ${\bf x}\preceq {\bf y}$, 
if $\sum_i x_i=\sum_i y_i$ and 
$\sum_{i=1}^k ({\bf x}^\downarrow)_i\le \sum_{i=1}^k ({\bf y}^\downarrow)_i$, 
for all $1\le k\le n$. 
Rado \cite{Rad52}  showed that $\Pi_{n-1}=\{{\bf x}\in\R^n:{\bf x}\preceq{\bf w}_n\}$
and Landau \cite{Lan53} proved that 
${\bf w}\in\Z^n$ is a win sequence if and only if ${\bf w}\preceq{\bf w}_n$. 
Therefore, $\win(n)=\Z^n\cap \Pi_{n-1}$.

Let ${\bf 1}_n=(1,\ldots,1)\in\Z^n$. 
We will sometimes make a linear shift, 
and consider the {\it score sequence} 
\[
{\bf s}(T)
={\bf w}(T)-\frac{n-1}{2}{\bf 1}_n 
=\sum_{i>j}(w_{ij}-1/2)({\bf e}_i-{\bf e}_j). 
\]
This corresponds to awarding $\pm1/2$ point for each win/loss, 
and centers the permutahedron at ${\bf 0}_n=(0,\ldots,0)\in\Z^n$. 
In particular, 
we let 
\[
{\bf s}_n
={\bf w}_n-\frac{n-1}{2}{\bf 1}_n
\] 
denote the {\it standard score sequence}.

\subsection{Interchange graphs}

Although the set $\score(n)$ of all score sequences has a 
simple geometric description (shifted lattice points in $\Pi_{n-1}$), 
the set $\tour(n,{\bf s})$ of all tournaments with given 
score sequence ${\bf s}$
is combinatorially rich. It seems difficult to fully describe its
structure in general.  
For instance, precise asymptotics for the size of 
$\tour(n,{\bf s})$ are known only when  ${\bf s}$
is close to ${\bf 0}_n$; see Spencer 
\cite{Spe74} and McKay et al.\ \cite{McKay90,McKW96,IIMcK20}. 

If ${\bf s}(T)={\bf 0}_n$, 
we say that $T$ is {\it neutral}. 
The {\it cyclic triangle} $\Delta_c$ is the smallest
non-trivial neutral tournament. In this tournament, each of the three players
wins exactly one game against the other two, and so ${\bf s}(\Delta_c)={\bf 0}_3$. 
For a tournament $T$ with a copy $\Delta\subseteq T$ of $\Delta_c$, we let
$T*\Delta$ denote the tournament obtained from $T$ by {\it reversing}
the orientation of all directed edges in $\Delta$. Since $\Delta$ is neutral, 
it follows that ${\bf s}(T*\Delta)={\bf s}(T)$.

Brualdi and Li \cite{BL84} studied the {\it interchange graph} 
$\ig(n,{\bf s})$ 
which encodes the combinatorics of $\tour(n,{\bf s})$. 
In this graph, 
there is a vertex $v(T)$ for each 
$T\in \tour(n,{\bf s})$. 
Vertices $v(T_1),v(T_2)$ are joined by an edge
if $T_2=T_1* \Delta$. 
The graph 
$\ig(n,{\bf s})$ is connected. In this sense, 
$\Delta_c$ {\it generates} $\tour(n,{\bf s})$. 

A fundamental property of $\ig(n,{\bf s})$ is that it is regular.  
In other words, any two tournaments 
with the same score sequence have the same number of 
cyclic triangles.
There are many proofs of this fact; see 
Moon \cite{Moo68}. 
The simplest way is to 
observe that, if ${\bf w}=(w_1,\ldots,w_n)$ is the 
win sequence, then there are precisely $\sum_i{w_i\choose2}$
non-cyclic triangles, since each such triangle
has a unique player that wins against the other two. 

Recalling that
${\bf s}={\bf w}-\frac{n-1}{2}{\bf 1}_n$, 
we observe, apparently for the first time, that 
the degree $d(n,{\bf s})$ of $\ig(n,{\bf s})$
can be described in terms of distances in  
$\Pi_{n-1}$. Specifically, 
\begin{equation}\label{E_triangles}
d(n,{\bf s})=\frac{\| {\bf s}_n\|^2-\| {\bf s}\|^2}{2},
\end{equation}
where 
$\|{\bf x}\|^2
=\sum_i x_i^2$. 

Roughly speaking, $\|{\bf s}_n\|$ is the radius of 
$\Pi_{n-1}'=\Pi_{n-1}-\frac{n-1}{2}{\bf 1}_n$. 
All of its vertices (permutations of ${\bf s}_n$)
are at the same distance from its center. 
The degree $d(n,{\bf s})$ of $\ig(n,{\bf s})$ increases  
as ${\bf s}$ moves 
closer to the centre ${\bf 0}_n$ of $\Pi_{n-1}'$. Since $\Delta_c$ is neutral, 
it is natural
to expect tournaments which are closer to being neutral
to contain more copies of $\Delta_c$. In the other extreme, 
vertices of $\Pi_{n-1}'$ correspond to acyclic tournaments
with no copies of $\Delta_c$.

\section{Purpose}

{\it Coxeter tournaments} were introduced in  \cite{KS23}. 
Such tournaments
can have collaborative
and solitaire games, in addition to the usual
competitive games in classical tournaments. 
The motivation 
for studying these tournaments 
is their connection with the 
{\it Coxeter permutahedra} $\Pi_\Phi$ 
in
Ardila, Castillo, Eur and Postnikov \cite{ACEP20}.

\begin{figure}[h!]
\centering
\includegraphics[scale=0.85]{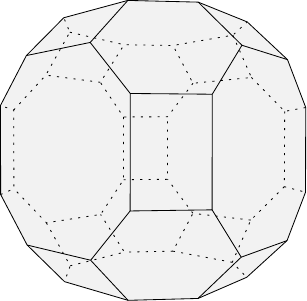}
\caption{A Coxeter permutahedron $\Pi_\Phi$   
of type $\Phi=C_3$.
}
\label{F_permC}
\end{figure}

The permutahedron $\Pi_{n-1}$
corresponds to the 
well-studied
root system of type $\Phi=A_{n-1}$. 
As in \cite{KS23}, we focus on the types $\Phi=B_n$, $C_n$ and $D_n$. 
Hyperplane descriptions of $\Pi_\Phi$ are obtained in \cite{ACEP20} 
in terms of submodular functions, generalizing the result of
Rado \cite{Rad52} mentioned above. 
Using these,
it is shown in \cite{KS23} that  
$\Pi_\Phi$ is the set of mean score sequences
of random Coxeter tournaments, generalizing a 
classical result of Moon \cite{Moo63}.

In this work, we study the combinatorics of 
deterministic Coxeter tournaments. 
We introduce 
Coxeter interchange graphs $\ig(\Phi,{\bf s})$
and prove an analogue (see \cref{T_MainCount}) of \eqref{E_triangles}
for their degrees $d(\Phi,{\bf s})$.
Along the way, we will also obtain an analogue (see \cref{T_MainLan}) of 
Landau's classical theorem \cite{Lan53},
characterizing the set 
$\score(\Phi)$
of points in 
$\Pi_\Phi$ that are score sequences of
Coxeter tournaments. 

The main difficulty
in extending \eqref{E_triangles} is that the graphs
$\ig(\Phi,{\bf s})$ involve several types of {\it generators;}
see \cref{F_genB,F_genC,F_genD}.
In type $A_{n-1}$ there is only one generator
$\Delta_c$, which makes the situation 
much simpler. 
Although all of these additional generators 
can be seen as neutral cycles, in the sense
of \cite{Zas82}, the Coxeter setting is significantly
more complicated due to interactions between
the various types of generators. 

The current work also forms the basis
for \cite{BKMP23},
which studies random walks
on Coxeter interchange graphs. 
Other structural properties of 
$\ig(\Phi,{\bf s})$ are proved in 
\cite{BKMP23}, such as their connectivity
and diameter. 

\subsection{Outline}
Coxeter tournaments are 
defined in \cref{S_Cox_tour}.
See \cref{S_results} for our results
and \cref{S_construct,S_count} for the proofs.

\section{Coxeter tournaments}
\label{S_Cox_tour}

Root systems are discussed in 
\cite[Section 2]{KS23} and in many standard texts. 
We focus on the 
infinite families of 
types $\Phi=B_n$, $C_n$ and $D_n$,
given by 
\begin{align*}
A_{n-1} &= \{ {\bf e}_i - {\bf e}_j : i \not = j \in [n]\},\\
B_n &= \{\pm {\bf e}_i \pm {\bf e}_j : i \not = j \in [n]\} \cup \{\pm {\bf e}_i : i \in [n]\},\\
C_n &= \{\pm {\bf e}_i \pm {\bf e}_j : i \not = j \in [n]\} \cup \{\pm 2{\bf e}_i : i \in [n]\},\\
D_n &= \{\pm {\bf e}_i \pm {\bf e}_j : i \not = j \in [n]\}.
\end{align*}
We select the following  
{\it positive systems} $\Phi^+$ for these
$\Phi$:  
\begin{align*}
A_{n-1}^+&=\{{\bf e}_i-{\bf e}_j:i>j\in[n]\},\\
B_{n}^+&=\{{\bf e}_i \pm {\bf e}_j : i > j \in [n]\} \cup \{{\bf e}_i : i \in [n]\}, \\
C_{n}^+&=\{{\bf e}_i \pm {\bf e}_j : i > j \in [n]\} \cup \{2{\bf e}_i : i \in [n]\}, \\
D_{n}^+&=\{{\bf e}_i \pm {\bf e}_j : i > j \in [n]\}.
\end{align*}

\begin{defn}
\label{D_SignedGraph}
To each $S\subseteq \Phi^+$, we associate a 
{\it signed $\Phi$-graph} $\cS$ with vertex set $[n]$
and edge set $E(\cS)$ which includes a
\begin{itemize}
\item {\it negative edge} $e_{ij}^-$ for each ${\bf e}_{ij}^-={\bf e}_i- {\bf e}_j\in S$, 
\item {\it positive edge} $e_{ij}^+$  for each ${\bf e}_{ij}^+={\bf e}_i+ {\bf e}_j\in S$, 
\item {\it half edge} $e_i^h$ for each ${\bf e}_i^h={\bf e}_i\in S$, 
\item {\it loop} $e_i^\ell$ for each ${\bf e}_i^\ell=2{\bf e}_i\in S$. 
\end{itemize}
Since edges in $\cS$ and vectors in $S$ are in bijective correspondence, 
we will denote the vector corresponding to an edge
$e$ by ${\bf e}$ (often also with sub and superscripts, as above, when 
appropriate). 
\end{defn}

What we call negative/positive edges above are 
positive/negative edges in Zaslavsky \cite{Zas91}.
The above convention (adopted in \cite{KS23}) 
is more natural in our current context, 
since negative/positive edges 
will correspond to competitive/collaborative games
in a Coxeter tournament.

Signed graphs of 
all types 
$B_n$, $C_n$ and $D_n$ 
have negative and positive edges. However, 
only $B_n$-graphs have half edges, and only 
$C_n$-graphs have loops. 

\begin{defn}
We let $\cK_\Phi$ denote the 
{\it complete $\Phi$-graph}, associated with the entire positive
system $S=\Phi^+$. 
\end{defn}

Note that classical graphs $G$ correspond to 
$A_{n-1}$-graphs $\cS$ with only 
negative edges. 
Therefore, the complete graph $K_n$ corresponds to $\cK_{A_{n-1}}$.

Recall that a tournament is an orientation of $K_n$. 
Similarly, a Coxeter tournament of type $\Phi$ is an orientation of 
$\cK_\Phi$. 

\begin{defn}
\label{D_CoxTour}
A {\it Coxeter tournament} $\cT$ on $\cK_\Phi$ is an orientation of 
the edges $e\in \Phi^+$. In this context, we refer to 
the edges $e$ as {\it games}. 
Formally, $\cT=\{w_e:e\in \Phi^+\}$, for some 
choice of $w_e\in\{0,1\}$. 
The {\it score sequence} 
of $\cT$ is given by   
\begin{equation}\label{E_bfs}
{\bf s}(\cT)=\sum_{e\in \Phi^+} (w_e-1/2) {\bf e}. 
\end{equation}
\end{defn}

In other words, 
the values $w_e$ determine the orientation of the edges $e$. 
When $w_e=1$ the vector ${\bf e}$ contributes $+{\bf e}/2$ to 
${\bf s}(\cT)$, and when 
$w_e=0$ the contribution is $-{\bf e}/2$. 
More explicitly, 
points are awarded as follows. 
\begin{itemize}
\label{D_CoxeterGames}
\item 
Negative edges 
$e_{ij}^-$ are {\it competitive games.}
One player wins and the other loses $1/2$ point, 
contributing $(w_{ij}^--1/2){\bf e}_{ij}^-$ to ${\bf s}$. 
\item 
Positive edges 
$e_{ij}^+$ are {\it collaborative games.}
Both players win or lose $1/2$ point, 
contributing $(w_{ij}^+-1/2){\bf e}_{ij}^+$ to ${\bf s}$. 
\item 
Half edges 
$e_i^h$ are {\it (half edge) solitaire games.}
One player wins or loses $1/2$ point, 
contributing $(w_i^h-1/2){\bf e}_i^h$ to ${\bf s}$.
\item 
Loops
$e_i^\ell$ are {\it (loop) solitaire games.}
One player wins or loses $1$ point, 
contributing $(w_i^\ell-1/2){\bf e}_i^\ell$ to ${\bf s}$.
\end{itemize}

Let us stress that, in all types $B_n$, $C_n$ and $D_n$, 
both games $e_{ij}^\pm$ are present for each pair $i,j$, and 
each such game is associated with its own $w_{ij}^\pm$. 
In types $B_n$ and $C_n$, there is also a solitaire game for each $i$.  
The difference is that solitaire games are worth twice as many
($\pm1$ rather than $\pm1/2$) points 
in type 
$C_n$ than in $B_n$, since 
${\bf e}_i^\ell=2{\bf e}_i$ and 
${\bf e}_i^h={\bf e}_i$.

The {\it standard score sequence,} 
corresponding to the Coxeter tournament on $\cK_\Phi$ in which 
all $w_e=1$, 
is denoted by
\begin{equation}\label{E_weylvec}
{\bf s}_\Phi=\sum_{{\bf e}\in\Phi^+} {\bf e}/2. 
\end{equation}

It will sometimes be useful to extend the notion of a 
score sequence ${\bf s}(\cX)$ 
to orientations $\cX$ of  
incomplete signed graphs $\cS$. 
We do so in the natural way, by summing over $E(\cS)$
rather than all of $\Phi^+$ in \eqref{E_bfs}.

\section{Results}
\label{S_results}

\subsection{Score sequences}
\label{S_MainLan}

For ${\bf x}, {\bf y}\in\R^n$, we
write ${\bf x}\preceq_w {\bf y}$, 
if $\sum_{i=1}^k ({\bf x}^\downarrow)_i\le \sum_{i=1}^k ({\bf y}^\downarrow)_i$, 
for all $1\le k\le n$. 
This notion is called {\it weak sub-majorization} in \cite{MOA11}.
For ${\bf x}\in\R^n$, we let  
$|{\bf x}|=(|x_1|,\ldots,|x_n|)$. 

Recall that the permutahedron $\Pi_{n-1}$
is the convex hull of all permutations
of the standard win sequence ${\bf w}_n$. 
In the Coxeter setting it is more convenient to work with score sequences. 
The Coxeter permutahedron $\Pi_\Phi$ is the convex hull 
of the orbit of 
the standard score sequence
${\bf s}_\Phi$, under the action of the Weyl group 
associated with the root system of type $\Phi$. 
In \cite{ACEP20}, Rado's theorem \cite{Rad52}
is generalized, showing 
$\Pi_\Phi=\{{\bf x}\in\R^n: |{\bf x}|\preceq_w {\bf s}_\Phi\}$.
In \cite{KS23}, it is shown that $\Pi_\Phi$ is the set of 
{\it mean} score sequences of {\it random} Coxeter tournaments. 

Our first result characterizes 
the set $\score(\Phi)$ of all score
sequences of (deterministic) Coxeter tournaments.

\begin{thm}
\label{T_MainLan}
Let $\Phi$ be a 
root system of type
$B_n$, $C_n$ or $D_n$. 
Then ${\bf s}\in\R^n$ is a score sequence  
${\bf s}={\bf s}(\cT)$ 
of some Coxeter tournament $\cT$
on the complete $\Phi$-graph $\cK_\Phi$
if and only if  $|{\bf s}| \preceq_w {\bf s}_\Phi$
and 
\begin{itemize}[nosep]
\item in $B_n$: ${\bf s}\in(\Z+1/2)^n$, 
\item in $C_n$: ${\bf s} \in \Z^n$ and $\sum_{i=1}^n s_i \equiv 
{n \choose 2} + n \Mod{2}$,
\item in $D_n$: ${\bf s} \in \Z^n$ and $\sum_{i=1}^n s_i \equiv 
{n \choose 2} \Mod{2}$. 
\end{itemize}
\end{thm}

\noindent {\it Remarks:}

\begin{enumerate}

\item 
Our proof is constructive, 
in that it shows how to {\it find} a Coxeter tournament 
$\cT$ with given score sequence ${\bf s}\in\score(\Phi)$.

\item 
The quantities ${n\choose2}+n$
and ${n\choose2}$  
are equal to $\sum_i({\bf s}_\Phi)_i$ in $\Phi=C_n$ and $D_n$. There is
no mod condition in $B_n$ due to the $1/2$ point solitaire games
in this type. 

\item
In \cite{Zas91} (see after Corollary 4.2 therein) 
Zaslavsky alludes to a 
characterization  of the score sequences on 
general signed graphs $\cS$, however, 
to the best of our knowledge, 
no such details have appeared. 
\cref{T_MainLan}
gives such a characterization in the 
case that $\cS=\cK_\Phi$ is complete. 

\end{enumerate}

\subsection{Coxeter interchange graphs}
\label{S_MainCount}

Our main result gives information about the set
$\tour(\Phi,{\bf s})$ of Coxeter tournaments
with a given score sequence ${\bf s}$. 

Recall that, in the classical setting, the set $\tour(n,{\bf s})$ is generated
by the cyclic triangle $\Delta_c$. 
We identify the generators in the Coxeter settings
of types $\Phi=B_n$, $C_n$ and $D_n$, and prove an analogue of 
\eqref{E_triangles} above. 
The generators are the smallest, non-trivial neutral 
Coxeter sub-tournaments. Each of them have three games. 

In $D_n$, in addition to $\Delta_c$, we require 
the {\it balanced triangle} $\Delta_b$, 
involving 
a collaborative win, collaborative loss, and competitive game, 
directed towards the collaborative win, 
as in \cref{F_genD}.
When depicting Coxeter tournaments, we  
draw competitive games as directed edges, directed away from the winner,
and we draw collaborative wins/losses as solid/dotted lines. 

\begin{figure}[h!]
\centering
\includegraphics[scale=1.15]{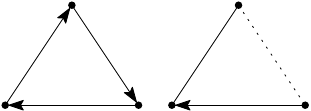}
\caption{Cyclic and balanced triangles
$\Delta_c$ and $\Delta_b$.}
\label{F_genD}
\end{figure}

In $B_n$ and $C_n$, we require 
$\Delta_c$, $\Delta_b$, and more. 
In $B_n,$ there are three {\it neutral pairs}
$\Omega_1$, $\Omega_2$ and $\Omega_3$, as in \cref{F_genB}. 
In  $C_n$, there are two {\it neutral clovers}
$\Theta_1$ and $\Theta_2$, as in \cref{F_genC}. 
We draw solitaire games as half edges (with only one endpoint), 
directed away/toward  
their endpoint if won/lost. In type $C_n$, we draw 
solitaire wins/loses as solid/dotted loops. 
We call the various $\Delta_c$, $\Delta_b$, $\Omega_i$, $\Theta_j$ {\it generators}.
\begin{figure}[h!]
\centering
\includegraphics[scale=1.15]{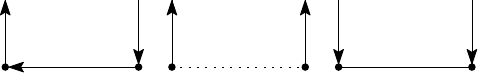}
\caption{Neutral pairs
$\Omega_1$, $\Omega_2$ and $\Omega_3$. 
}
\label{F_genB}
\end{figure}

\begin{figure}[h!]
\centering
\includegraphics[scale=1.15]{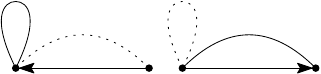}
\caption{Neutral clovers 
$\Theta_1$ and $\Theta_2$.
}
\label{F_genC}
\end{figure}

The {\it reversal} $\cX^*$ 
of an orientation $\cX=\{w_e:e\in E(\cS)\}$ (with all $w_e\in\{0,1\}$) 
of a (possibly incomplete) signed graph $\cS$ 
is obtained by 
reversing the orientation of 
all games in $\cX$. 
That is, 
$\cX^*=\{w_e^*:e\in E(\cS)\}$, where $w_e^*=1-w_e$. 
More generally, we let $\cX*\cY$ denote the reorientation 
obtained from $\cX$ by reversing the orientation of the games in 
some $\cY\subseteq \cX$. 
Note that 
$\cX^*=\cX*\cX$.

We say that $\cX$ is {\it neutral} if ${\bf s}(\cX)={\bf 0}_n$. 
Note that the cyclic and balanced triangles $\Delta_c$ and $\Delta_b$
and the neutral pairs $\Omega_i$ and clovers $\Theta_i$ are all neutral. 
As such, reversing a copy of a generator  in a Coxeter tournament does not 
change its score sequence.

We generalize the interchange graphs of 
Brualdi and Li \cite{BL84}. 

\begin{defn}
\label{D_CoxIG}
The {\it Coxeter $\Phi$-interchange graph} $\ig(\Phi,{\bf s})$
has a vertex $v(\cT)$ for each 
$\cT\in \tour(\Phi,{\bf s})$. Vertices $v(\cT_1),v(\cT_2)$
are neighbors if $\cT_2=\cT_1*\cG$ for some 
type $\Phi$ generator $\cG$ in $\cT_1$. When $\cG$ is a clover, 
we add a double edge between neighbors, 
and a single edge otherwise. 
\end{defn}

See \cref{F_IGs} above
for examples.

As we will see, double edges, associated with clover reversals,
are added to maintain
regularity. This is perhaps not unexpected, 
since clovers are the only generators 
that include loop games (worth 
$\pm1$ point rather than $\pm1/2$). 

\begin{thm}
\label{T_MainCount}
Let $\Phi=B_n$, $C_n$ or $D_n$ and ${\bf s}\in\score(\Phi)$. 
Then the Coxeter interchange graph $\ig(\Phi,{\bf s})$ is regular, 
with degree
\[
d(\Phi,{\bf s})
=\frac{\|{\bf s}_\Phi\|^2-\| {\bf s} \|^2}{2},
\]
where ${\bf s}_\Phi$ is the standard score sequence.
\end{thm}

\noindent {\it Remarks:}

\begin{enumerate}

\item 
More explicitly, 
$d(\Phi,{\bf s})$ counts the number  $[\cG]$ of copies of type $\Phi$ 
generators $\cG$ in Coxeter tournaments $\cT\in\tour(\Phi,{\bf s})$ as follows: 
\[
d(\Phi,{\bf s})
=[\Delta_c]+[\Delta_b]+
\begin{cases}
[\Omega_1]+[\Omega_2]+[\Omega_3], & \text{if }\Phi=B_n,\\
2([\Theta_1]+[\Theta_2]), & \text{if }\Phi=C_n.\\
\end{cases}
\]
Recall that when $\Phi=A_{n-1}$, we simply have $d(A_{n-1},{\bf s})=[\Delta_c]$. 
The ``weights'' 2 in type $C_n$ correspond to the double edges
associated with clovers. 

\item 
In \cite{Zas91} (see Figures 1 and 2 therein) 
oriented edges in oriented signed graphs are 
represented using half edges. In this representation (called the {\it Z-frame} in \cite[\S4.1]{BKMP23}), each vertex in 
the generators in \cref{F_genB,F_genC,F_genD} (and indeed any other neutral 
structure) has the same number of half edges coming in as going out. 
From this point of view, all generators in 
\cref{F_genB,F_genC,F_genD} are neutral cycles; indeed, they are 
the smallest such cycles. 
That being said,  
generalizing \eqref{E_triangles} is far from straightforward 
due to delicate interactions between the various types of generators. 
On the other hand, in type $A_{n-1}$
there is only one generator (the cyclic triangle $\Delta_c$), 
and so to prove  \eqref{E_triangles} one only needs
to count 
non-cyclic triangles, which is very simple.

\item 
We give two proofs of \cref{T_MainCount}. Our first proof 
in \cref{S_count_alg} 
is algebraic. 
Our second proof is combinatorial, 
and perhaps more intuitive. 
Specifically, in \cref{S_embed} we 
show how to embed 
a Coxeter tournament ${\cT}$ as a certain graph tournament $T=T(\cT)$, in such a way
that its score sequence and the various copies of 
generators in $\cT$ are recoverable. 
In this way, \cref{T_MainCount} 
can (rather elaborately) be 
derived from \eqref{E_triangles}.

\item Finally, we note that a version
of \cref{T_MainCount} (with appropriately chosen ``weights'' in 
$d(\Phi,{\bf s})$) holds
for the exceptional types $E_6$, $E_7$, $E_8$, $F_4$ and $G_2$. 
In these types, just as in types $A_{n-1}$, $B_n$, $C_n$, $D_n$, 
the generators consist of triples of (oriented) roots which sum up to $0$. 
Versions for the exceptional types follow by a straightforward 
adaptation of our proof in \cref{S_count_alg}. 
We leave the details to the interested reader. 

\end{enumerate}

\section{Characterizing score sequences}
\label{S_construct}

In this section, we prove \cref{T_MainLan}. 
In \cref{S_Nec}, we show that the conditions are necessary. 
Then, in \cref{S_CharCD}, we show that these conditions are sufficient
when $\Phi=C_n$ or $D_n$. The proof is constructive. 
First, we show that it suffices to 
construct a Coxeter tournament with score sequence
$|{\bf s}|$, with non-negative entries. Once such a 
Coxeter tournament has been constructed,
there is a way of modifying it, so as to obtain one with score sequence ${\bf s}$. 
We show that constructing a Coxeter tournament with a 
non-negative score sequence eventually reduces to constructing a 
related, classical graph tournament. 
Finally, in \cref{S_CharB}, when $\Phi=B_n$, we establish 
a reduction to the 
type $D_n$ construction. 

We note that the standard score sequence 
in type $\Phi$ satisfies 
\begin{equation}\label{E_sw}
{\bf s}_\Phi={\bf w}_n+\delta_\Phi {\bf 1}_n,
\end{equation}
where ${\bf w}_n=(0,1,\ldots,n-1)$ is the standard win sequence
of type $A_{n-1}$ and the additional scalar factor 
\[
\delta_\Phi
=
\begin{cases}
1/2&\Phi=B_n,\\
1&\Phi=C_n,\\
0&\Phi=D_n,
\end{cases}
\]
accounts for solitaire games.

\subsection{Necessary conditions}
\label{S_Nec}
First, we show that the conditions in
\cref{T_MainLan} are necessary. 

Technically, the quantities 
$w_{ij}^\pm$ were only defined above for $i>j$. 
We extend these for $i < j$ by setting
$w_{ij}^- = 1 - w_{ji}^-$ and $w_{ij}^+ = w_{ji}^+$.

\begin{proof}[Proof of \cref{T_MainLan} (only if part)]
Suppose that  ${\bf s}\in\score(\Phi)$. Without loss of generality, 
assume $|s_1| \ge |s_2| \ge \cdots \ge |s_n|$. 
Let
\[
s_{ij}^\pm 
= w_{ij}^\pm -1/2
\]
be the points that $i$ earns  
from its competitive/collaborative game with $j$. 
Put $s_{ij} = s_{ij}^+ + s_{ij}^-$. 
Then the total number of points earned by $i$ is 
\[
s_i 
= \sum_{j \neq i} s_{ij} + \eps_i \delta_\Phi,
\]
for some $\eps_i=\pm1$. 
Note that for every pair $\{i,j\}$, 
exactly one of $s_{ij},s_{ji}=0$ and the other is $\pm 1$. 
Hence ${\bf s} \in (\Z + 1/2)^n$ if $\Phi = B_n$ 
and ${\bf s} \in \Z^n$ if $\Phi = C_n$ or $D_n$.
Further, if $\Phi = C_n$ or $D_n$, then
\[ 
\sum_{i=1}^n s_i 
\equiv
 \sum_{i<j} (s_{ij} + s_{ji}) + n\delta_\Phi
\equiv {n \choose 2} + n\delta_\Phi \Mod{2}.
\]

Finally, observe that
\[
|s_i| 
\le \sum_{j \neq i} |s_{ij}| + \delta_\Phi.
\]
It follows that
\[
\sum_{i=1}^k |s_i| 
\le 
\sum_{\substack{\ \ i,j\colon \\ i<j\le k}} (|s_{ij}| + |s_{ji}|) + 
\sum_{\substack{\ \ i,j\colon \\ i \le k < j}} |s_{ij}| + k\delta_\Phi.
\]
Since each $|s_{ij}| + |s_{ji}|=1$ and $|s_{ij}|\le 1$, it follows that 
\[
\sum_{i=1}^k |s_i| 
\le {k\choose2}+k(n-k)+k\delta_\Phi,
\]
and so, by \eqref{E_sw}, it follows that $|{\bf s}|\preceq_w {\bf s}_\Phi$,
as required. 
\end{proof}

\subsection{Sufficiency in $C_n$ and $D_n$}
\label{S_CharCD}

We begin with 
types $C_n$ and $D_n$. The conditions in these cases
are similar, and can be dealt with simultaneously. 
A number of preliminary results (\cref{L_construction_leq,L_construction_strongleq,L_construction_if_positive_done,L_construction_evenjumps,L_construction_Dfinal,L_construction_Cfinal}) are required. 

\begin{lemma}\label{L_construction_leq}
If ${\bf x}, {\bf y} \in \Z^n$ are such that ${\bf x} \preceq_w {\bf y}$, 
then there exists ${\bf z} \in \Z^n$ so that 
\[
{\bf x} \le {\bf z} \preceq {\bf y},
\] 
where ${\bf x} \le {\bf z}$ means that $x_i \le z_i$ for all $i\in[n]$.
\end{lemma}

See \cite[Corollary 1.8]{And89} for a proof in the case that 
${\bf x}, {\bf y}, {\bf z} \in \R^n$.

\begin{proof}
If $\sum_i x_i=\sum_i y_i$ then ${\bf x}\preceq {\bf y}$, 
and we take ${\bf z}={\bf x}$. 
Otherwise, let ${\bf x}^{(1)}$ be the sequence obtained from ${\bf x}$ 
by increasing its smallest coordinate by~$1$. 
It can be seen that ${\bf x}^{(1)}\preceq_w {\bf y}$. 
Therefore, repeating this procedure until 
$\sum_i x^{(k)}_i= \sum_i y_i$, we find the requisite  ${\bf z}={\bf x}^{(k)}$. 
\end{proof}

\begin{lemma}\label{L_construction_strongleq}
If ${\bf x}, {\bf y}, {\bf z} \in \Z^n$ are such that 
\[
{\bf x} \le {\bf z} \preceq {\bf y}  \quad \text{and} \quad \sum_{i=1}^n x_i \equiv \sum_{i=1}^n y_i \Mod{2},
\]
then there exists ${\bf z}' \in \Z^n$ so that
\[ 
{\bf x} \le {\bf z}' \preceq {\bf y} \quad \text{and} \quad z'_i \equiv x_i \Mod{2}
\]
for every $i \in [n]$.
\end{lemma}

\begin{proof}
Without loss of generality 
we may
assume that $z_1 \ge z_2 \ge \cdots \ge z_n$. 
Let 
\[
I  
= \{i \in [n]\colon x_i\not \equiv z_i \Mod{2}\} 
= \{i_1 < i_2 < \cdots < i_k\}.
\]
Since $\sum z_i = \sum y_i$ and 
$\sum x_i \equiv \sum y_i \Mod{2}$, 
it follows that $|I|=k$ is even. 
Put 
\[ 
z'_i 
= \begin{cases} 
z_i-1 &i\in\{i_1,i_3,\ldots,i_{k-1}\}, \\
z_i+1 &i\in\{i_2,i_4,\ldots,i_{k}\},\\
z_i &i\notin I.
\end{cases} 
\]
Note that $\sum z'_i = \sum z_i = \sum y_i$. 
In fact, ${\bf z}' = {\bf z} + {\bf u}$, 
for some ${\bf  u}\in\{0,\pm1\}^n$ with all partial sums 
$\sum_{i=1}^k u_i\le 0$, with equality when $k=n$. 
Hence, ${\bf z}'={\bf z} + {\bf u}\preceq {\bf y}+{\bf 0}_n={\bf y}$.  
Finally, observe that,   
in obtaining ${\bf z}'$ from ${\bf z}$, we have 
only modified the coordinates of ${\bf z}$ for which $x_i < z_i$. Therefore, 
since ${\bf x}\le {\bf z}$, 
it follows that 
${\bf x} \le {\bf z}'$.
\end{proof}

\begin{lemma}\label{L_construction_if_positive_done}
Suppose that $\Phi=C_n$ or $D_n$ and let ${\bf s} \in \Z^n$. 
If $|{\bf s}|\in\score(\Phi)$
then ${\bf s}\in \score(\Phi)$. 
\end{lemma}

\begin{proof}
Let $\cT$ be a Coxeter tournament with score sequence $|{\bf s}|$. 
We will construct $\cT'$ with score sequence ${\bf s}$ by, roughly speaking, 
reversing some of the games involving 
weak players $i$ in $\cT$ such that $s_i<0$. 
By relabelling players, we may assume that $s_i<0$ for all $i\le i_0$
and $s_i\ge0$ for $i>i_0$. 
We let $w_{ij}^\pm$ and $w_{ij}^\ell$ denote the games in $\cT$, 
and $w_{ij}'^\pm$ and $w_{ij}'^\ell$ the games in $\cT'$. 

Consider player $i$ for $i \le i_0$. 
Let $A_i$ be the set of all players $j\neq i$ 
which earn a net $0$ number of points from player $i$. 
Note that this occurs if and only if 
$w_{ij}^+ = w_{ij}^- = 1$ (resp.\ $=0$), in which case
player $i$ earns net $1$ (resp.\ $-1$) point from player $j$. 
We reverse the outcome of all such games. That is, we put 
$w_{ij}'^\pm=1-w_{ij}^\pm$
for all $j\in A_i$. Note that, after this reversal, player $j$ 
will still earn net $0$ points from player $i$. 
Also, if $\Phi=C_n$, 
we reverse the outcome of its solitaire game, $w_i'^\ell= 1-w_i^\ell$.
All other games 
are unchanged. 

Note that, by construction, the score of each player $i\le i_0$ switches from 
$|s_i|$ to $-|s_i|=s_i$, 
and the score of each other player $i>i_0$ is unchanged, $|s_i|=s_i$. 
Therefore, $\cT'$ has score sequence ${\bf s}$, as desired. 
\end{proof}

\begin{lemma}\label{L_construction_evenjumps}
Suppose that $\Phi=C_n$ or $D_n$. 
If ${\bf z}\in\score(\Phi)$ and ${\bf s} \in \Z^n$
satisfy
\[ 
{\bf 0}_n\le {\bf s}\le {\bf z} \quad \text{and} \quad s_i \equiv z_i \Mod{2}
\]
for all $i\in[n]$, 
then ${\bf s}\in\score(\Phi)$. 
\end{lemma}

\begin{proof}
Suppose that $\cT$ has score sequence ${\bf z}$. 
We will modify the games of $\cT$ to obtain a tournament $\cT'$ with 
score sequence ${\bf s}$. 

For each player $i$, let $P_i$ be the set of players $j$ from which player $i$ 
gains $1$ point. That is, $j\in P_i$ if player $i$ wins 
both of its competitive and collaborative
games with player $j$. Note that player $j$ 
gains net 0 points from $i$. If $\Phi=C_n$, then we also include $\infty\in P_i$ if 
player $i$ wins its solitaire game, $w_i^\ell=1$.
In this context, we think of a solitaire win as equivalent to 
competitive and collaborative wins with player $\infty$,
whose score is not included in ${\bf z}$. 
Observe that $z_i \le |P_i|$, since, by the choice of $P_i$, 
the net contribution of points from all other games
towards $z_i$ is non-positive. 

To obtain $\cT'$, select an arbitrary subset $M_i \subseteq P_i$ of size $(z_i-s_i)/2$. 
We can do this since $z_i \equiv s_i \Mod{2}$ and $0\le s_i\le z_i$, so 
\[ 
\frac{z_i-s_i}{2} \le z_i \le |P_i|.
\]
Next, to obtain $\cT'$, 
we reverse the outcome of all games in 
$\cT$ between $i$ and players $j\in M_i$. 
Recall that, by the choice of $P_i$ each such reversal has a 
net $-2$ effect on the score of player $i$, and a 
net 0 effect
on the score of player $j$. 
As a result of this construction, in $\cT'$ each player $i$ scores 
\[
z_i-2\cdot \frac{z_i-s_i}{2}=s_i
\]
points in total. That is, $\cT'$ has score sequence ${\bf s}$, as desired. 
\end{proof}

\begin{lemma}\label{L_construction_Dfinal}
Let ${\bf s} \in \Z^n$ be such that 
${\bf s} \preceq {\bf w}_n$.
Then ${\bf s}\in\score(D_n)$.
\end{lemma}

\begin{proof}
First we note that
${\bf s}$ is attainable 
as a {\it win} sequence of a classical (type $A_{n-1}$) tournament $T$. 
We can construct a $D_n$-tournament $\cT$ with {\it score} sequence ${\bf s}$
by setting all $w_{ij}^+=1$ in $\cT$, and setting $w_{ij}^-=1$ in $\cT$ if and only if 
$w_{ij}=1$ in $T$, for $i > j$. 
That is, all collaborative games are won in $\cT$ and the results of competitive
games in $\cT$ are the same as in $T$. 
\end{proof}

\begin{lemma}\label{L_construction_Cfinal}
If ${\bf s} \in \Z^n$ and ${\bf s} \preceq {\bf s}_{C_n} = {\bf w}_n + {\bf 1}_n$, 
then ${\bf s} \in \score(C_n)$.
\end{lemma}

\begin{proof}
The conditions imply ${\bf s} - {\bf 1}_n \preceq {\bf w}_n$. 
By \cref{L_construction_Dfinal}, we can construct a $D_n$-tournament 
with score sequence ${\bf s} - {\bf 1}_n$. 
Adding a solitaire loop game win to every player 
yields a $C_n$-tournament with score sequence ${\bf s}$.
\end{proof}

With these results in hand, we can complete the proof of
\cref{T_MainLan} in the case $C_n$ and $D_n$. 

\begin{proof}[Proof of \cref{T_MainLan} (if part, types $C_n$ and $D_n$)]
Let $\Phi=C_n$ or $D_n$ and 
suppose that ${\bf s}$ satisfies 
\begin{itemize}[nosep]
\item $|{\bf s}| \preceq_w {\bf  s}_\Phi$, 
\item ${\bf s} \in \Z^n$,
\item $\sum_i s_i \equiv \sum_i ({\bf  s}_\Phi)_i={n\choose2}+\delta_\Phi n \Mod{2}$. 
\end{itemize}
We claim that 
there exists a Coxeter tournament $\cT$ with score sequence ${\bf s}$. 
To see this, first note that, by \cref{L_construction_if_positive_done}, 
we may assume that ${\bf 0}_n\le {\bf s}$, in which case $|{\bf s}|={\bf s}$. 
Then, by \cref{L_construction_leq,L_construction_strongleq}, it follows that 
there is a sequence ${\bf z}'$ such that
\[ 
{\bf 0}_n \le {\bf s} \le {\bf z}' \preceq {\bf s}_\Phi 
\quad \text{and} 
\quad s_i \equiv z_i' \Mod{2}
\]
for every $i \in [n]$. Hence, by \cref{L_construction_evenjumps}, 
it suffices to verify the existence of a Coxeter tournament 
with score sequence ${\bf z}'$, 
which follows by \cref{L_construction_Dfinal,L_construction_Cfinal}, 
completing the proof.
\end{proof}

The above proofs are constructive. See \cref{F_ex} for an example. 

\begin{figure}[h!]
\centering
\includegraphics[scale = 0.85]{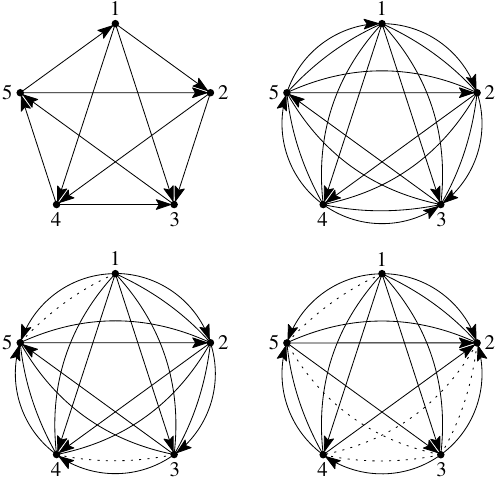}
\caption{
To find a 
$D_5$-tournament with ${\bf s}=(3,-2,-1,0,0)$
(bottom right), it suffices by \cref{L_construction_if_positive_done}
to find a $D_5$-tournament with score sequence $|{\bf s}|$
(bottom left). 
To this end, using \cref{L_construction_leq,L_construction_strongleq}, 
we find 
an $A_{4}$-tournament with ${\bf z}' = (3,2,1,2,2)$
(top left) and then a $D_5$-tournament with score sequence ${\bf z}'$
(top right), where 
${\bf z} = (3,2,2,2,1)$.}
\label{F_ex}
\end{figure}

\subsection{Sufficiency in $B_n$}
\label{S_CharB}

Finally, we address the remaining case of type $B_n$. 
The overall idea is to reduce to type $D_n$. 
Given ${\bf s}$, we first determine the result
of the half edge solitaire games. 
We begin by 
attempting to let each {\it strong/weak} 
player $i$ with a positive/negative 
$s_i$ win/lose
its solitaire game. 
If the remaining sequence ${\bf s}'$ 
of points yet to win
is in $\score(D_n)$ then we are done. 
Otherwise, we reverse the solitaire game
for the most neutral player, and argue that 
the adjusted ${\bf s}''$ is in $\score(D_n)$.

\begin{proof}[Proof of \cref{T_MainLan} (if part, type $B_n$)]
Let $\Phi=B_n$ and suppose that ${\bf s}$ satisfies 
\begin{itemize}[nosep]
\item $|{\bf s}|\preceq_w {\bf s}_\Phi={\bf w}_n+{\bf 1}_n/2$, 
\item ${\bf s}\in(\Z+1/2)^n$.
\end{itemize}
Without loss of generality, we may assume 
that $|s_1|\le\cdots\le |s_n|$. 

Consider the sequences
\begin{align*}
{\bf s}' &= (s_1 - \eps_1/2,s_2 - \eps_2/2, \ldots, s_n - \eps_n/2),\\
{\bf s}'' &= (s_1 + \eps_1/2,s_2 - \eps_2/2, \ldots, s_n - \eps_n/2),
\end{align*}
where $\eps_i = \sgn(s_i) \neq 0$.  
As discussed above, these sequences can be interpreted as 
sequences of points yet to win after all solitaire games
have been played. Therefore, it suffices to show that
either ${\bf s}'$ or ${\bf s}''$ is in $\score(D_n)$. 

If $\sum s_i' \equiv {n \choose 2} \Mod{2}$, 
then by \cref{T_MainLan} it follows that  
${\bf s}'\in \score(D_n)$. 
Otherwise, 
if $\sum s_i' \not\equiv {n \choose 2} \Mod{2}$, then  $\sum |s_i'| < {n\choose2}$. 
Therefore, $|{\bf s}''| \preceq_w {\bf w}_n$. 
Furthermore, note that 
$\sum s_i'' \equiv 1+\sum s_i'\equiv {n \choose 2} \Mod{2}$. 
Hence ${\bf s}''\in \score(D_n)$.  
\end{proof}

\section{Counting generators}
\label{S_count}

In this section, we give two proofs of the Coxeter counting 
formulas \cref{T_MainCount}. The first proof in 
\cref{S_count_alg} is algebraic. The second proof
in \cref{S_embed} is combinatorial, via a sequence
of constructive embeddings.

\subsection{Counting algebraically}
\label{S_count_alg}

Let $\cT=\{w_e:e\in E(\cK_\Phi)\}$ be a Coxeter tournament
on $\cK_\Phi$, with score sequence ${\bf s}={\bf s}(\cT)$. 
For ${\bf e}\in \Phi^+$, put 
$\hat  {\bf e}=(2w_e-1){\bf e}$. Then 
\[
{\bf s}=\sum_{{\bf e}\in\Phi^+}\hat  {\bf e}/2.
\]
Recall (see \eqref{E_weylvec} above) that the standard score sequence ${\bf s}_\Phi$
corresponds to the tournament in which all $w_e=1$. 

In this section, we show that 
\[
\frac{\| {\bf s}_\Phi\|^2-\| {\bf s}\|^2}{2} 
\]
counts the number of type $\Phi$ generators in $\cT$, 
with neutral clovers (only in $C_n$) counted twice. 

To this end, first note that all $\| {\bf e} \|^2=\|  \hat  {\bf e} \|^2$, 
since $\hat  {\bf e}=\pm{\bf e}$. 
Hence 
\[
\frac{\| {\bf s}_\Phi\|^2-\| {\bf s}\|^2}{2}
=\frac{\left\| \sum {\bf e} \right\|^2-\left\| \sum \hat  {\bf e} \right\|^2}{8}
=\sum_{{\bf e}\neq {\bf e}'}
\frac{\langle{\bf e},{\bf e}' \rangle-\langle \hat  {\bf e}, \hat {\bf e}'\rangle}{8}. 
\]

In fact, it suffices to sum over ${\bf e}\neq{\bf e}'$ which involve 
exactly one common player. Indeed, if they have no common player, then clearly
$\langle{\bf e},{\bf e}' \rangle=0$, 
since then for every $i\in[n]$ at least one of ${\bf e},{\bf e}'$
has $i$th entry equal to $0$. 
On the other hand, if they have two players in common, 
then 
$\langle {\bf e}_{ij}^-,{\bf e}_{ij}^+ \rangle=0$.

\subsubsection{Counting in $D_n$}
In type $D_n$ there are no solitaire games.
After some rearranging of terms, 
it can be seen that 
\begin{equation}\label{E_sumalphas}
\sum_{{\bf e}\neq {\bf e}'}
\frac{\langle{\bf e},{\bf e}' \rangle-\langle \hat  {\bf e}, \hat {\bf e}'\rangle}{8}
=\sum_{i>j>k} \frac{\psi_{ijk}-\hat \psi_{ijk}}{4},
\end{equation}
where
\begin{align*}
\psi_{ijk}
&=
\langle {\bf e}_{ij}^-,{\bf e}_{ik}^-\rangle+\langle {\bf e}_{ij}^-,{\bf e}_{jk}^-\rangle+\langle {\bf e}_{ik}^-,{\bf e}_{jk}^-\rangle\\
&\enspace+
\langle {\bf e}_{ij}^-,{\bf e}_{ik}^+\rangle+\langle {\bf e}_{ij}^-,{\bf e}_{jk}^+\rangle+\langle {\bf e}_{ik}^+,{\bf e}_{jk}^+\rangle\\
&\enspace+
\langle {\bf e}_{ij}^+,{\bf e}_{ik}^-\rangle+\langle {\bf e}_{ij}^+,{\bf e}_{jk}^+\rangle+\langle {\bf e}_{ik}^-,{\bf e}_{jk}^+\rangle\\
&\enspace+
\langle {\bf e}_{ij}^+,{\bf e}_{ik}^+\rangle+\langle {\bf e}_{ij}^+,{\bf e}_{jk}^-\rangle+\langle {\bf e}_{ik}^+,{\bf e}_{jk}^-\rangle,
\end{align*}
and with 
$\hat  \psi_{ijk}$ similarly defined. 
Recalling that all $\| {\bf e} \|^2=\| \hat  {\bf e} \|^2$, it follows that 
\[
2(\psi_{ijk}-\hat   \psi_{ijk})=
\delta_{ijk}^{---}
+\delta_{ijk}^{-++}
+\delta_{ijk}^{+-+}
+\delta_{ijk}^{++-},
\]
where
\[
\delta_{ijk}^{\sigma_{ij}\sigma_{ik}\sigma_{jk}}
=\| {\bf e}_{ij}^{\sigma_{ij}}+{\bf e}_{ik}^{\sigma_{ik}}+{\bf e}_{jk}^{\sigma_{jk}}\|^2 
-\| \hat  {\bf e}_{ij}^{\sigma_{ij}}+\hat  {\bf e}_{ik}^{\sigma_{ik}}+\hat  {\bf e}_{jk}^{\sigma_{jk}}\|^2.
\]

We claim that the first term
\[
\delta_{ijk}^{---}
= 8\cdot [\Delta_c]_{ijk},
\]
where $[\Delta_c]_{ijk}$ counts the number of copies of the 
cyclic triangle $\Delta_c$ in the sub-tournament induced on vertices $\{i,j,k\}$.
To see this, 
we observe that, in all possible cases, 
$\| \hat  {\bf e}_{ij}^-+\hat  {\bf e}_{ik}^-+\hat  {\bf e}_{jk}^-\|^2=0$
if the three games  
form a
cyclic triangle, and otherwise 
$\| \hat  {\bf e}_{ij}^-+\hat  {\bf e}_{ik}^-+\hat  {\bf e}_{jk}^-\|^2=8$. 
See \cref{F_countDd}. 
In particular, for $i>j>k$, 
${\bf e}_{ij}^-+{\bf e}_{ik}^-+{\bf e}_{jk}^-=2({\bf e}_i-{\bf e}_k)$, 
and so 
$\| {\bf e}_{ij}^-+{\bf e}_{ik}^-+{\bf e}_{jk}^-\|^2=8$. 

\begin{figure}[h!]
\centering
\includegraphics[scale = 1]{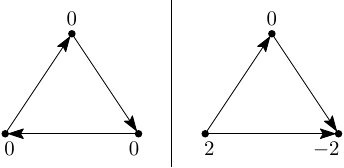}
\caption{Up to symmetries, there are two possible orientations of the 
negative edges on three vertices. The tournament on the left is neutral and 
the tournament on the right is not. The net wins by each vertex is given.}
\label{F_countDd}
\end{figure}

Similarly, it can be seen that 
\[
\delta_{ijk}^{-++}+\delta_{ijk}^{+-+}+\delta_{ijk}^{++-}
=8\cdot [\Delta_b]_{ijk}, 
\]
where $[\Delta_b]_{ijk}$ counts the number of copies of the balanced
triangle $\Delta_b$ on $\{i,j,k\}$. 
These 
terms account for the three possible locations of the negative edge in a balanced
triangle. 
See \cref{F_countDb}.

\begin{figure}[h!]
\centering
\includegraphics[scale = 1]{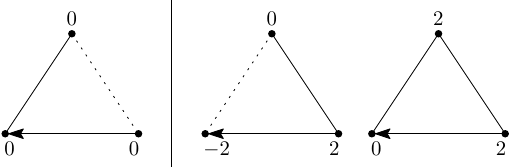}
\caption{Up to symmetries (including reversal), 
there are three possible orientations of one 
negative edge and two positive edges involving three vertices. 
The tournament on the left is neutral and 
the tournaments on the right are not. The net wins by each vertex is given.}
\label{F_countDb}
\end{figure}

Summing over $i>j>k$, we find that, in type $D_n$,  
\begin{equation}\label{E_Dcount}
\frac{\| {\bf s}_\Phi\|^2-\| {\bf s}\|^2}{2}
=[\Delta_c]+[\Delta_b]
\end{equation}
counts the total number of copies of $\Delta_c$ and $\Delta_b$
in $\cT$. 

\begin{remark}\label{Rem_CountD}
    If $\hat {\bf s}_{ijk} = \hat {\bf e}_{ij}^- + \hat {\bf e}_{ij}^+ + \hat {\bf e}_{ik}^- + \hat {\bf e}_{ik}^+ + \hat {\bf e}_{jk}^- + \hat {\bf e}_{jk}^+,$
    then
    \begin{equation*}
        \psi_{ijk} - \hat  \psi_{ijk} = \frac{\|{\bf s}_{D_3}\|^2 - \|\hat {\bf s}_{ijk}\|^2}{2}. 
    \end{equation*}
    Hence terms in the sum \eqref{E_sumalphas} 
    are equal to the quantity from \cref{T_MainCount},  applied to the sub-tournament 
    $\cT_{ijk}\subseteq \cT$
    induced by $\{i,j,k\}$. 
    Therefore, alternatively, to prove \cref{T_MainCount} 
    for type $D_n$ it is enough to show its validity 
    for all Coxeter tournaments of type $D_3$.
\end{remark}

\subsubsection{Counting in $B_n$}
In type $B_n$, there are half edge solitaire games. 
As such, 
in addition to the sum \eqref{E_sumalphas}, 
there is a term 
\begin{equation}\label{E_CountAlgBterm}
\sum_{i>j} \frac{\psi_{ij}^h-\hat \psi_{ij}^h}{4},
\end{equation}
where
\[
\psi_{ij}^h
=
\langle {\bf e}_i^h, {\bf e}_{ij}^-\rangle 
+ \langle {\bf e}_i^h, {\bf e}_{ij}^+\rangle 
+ \langle {\bf e}_j^h, {\bf e}_{ij}^-\rangle 
+ \langle {\bf e}_j^h, {\bf e}_{ij}^+\rangle
=\langle {\bf e}_i^h+{\bf e}_j^h,{\bf e}_{ij}^-+{\bf e}_{ij}^+\rangle,
\]
with $\hat  \psi_{ij}^h$ similarly defined. 

There is a unique neutral pair on $\{i,j\}$ if and only if 
$\hat  \psi_{ij}^h=-2$, 
and there are no neutral pairs on $\{i,j\}$ if and only if $\hat  \psi_{ij}^h=2$. 
See \cref{F_countB}.
In particular, ${\bf e}_i^h+{\bf e}_j^h={\bf e}_i+{\bf e}_j$
and ${\bf e}_{ij}^-+{\bf e}_{ij}^+=2{\bf e}_i$, 
and so $\psi_{ij}^h=2$. Hence
\[
\sum_{i>j} \frac{\psi_{ij}^h-\hat \psi_{ij}^h}{4}
=[\Omega_1]+[\Omega_2]+[\Omega_3]
\]
counts the total number of neutral pairs in $\cT$. 
As such, in type $B_n$, 
\begin{equation}\label{E_Bcount}
\frac{\| {\bf s}_\Phi\|^2-\| {\bf s}\|^2}{2}
=[\Delta_c]+[\Delta_b]+[\Omega_1]+[\Omega_2]+[\Omega_3],
\end{equation}
as claimed. 

\begin{figure}[h!]
\centering
\includegraphics[scale = 1]{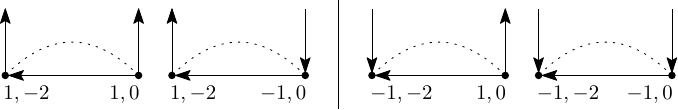}
\caption{Up to symmetries (including reversal), there are four $B_2$-tournaments. 
The tournaments on the left contain a neutral pair and 
the tournaments on the right do not. The net wins by each vertex from its solitaire
and competitive/collaborative games are given.}
\label{F_countB}
\end{figure}

\begin{remark} If 
    $\hat {\bf s}_{ij} = \hat {\bf e}_{ij}^- + \hat {\bf e}_{ij}^+ + \hat {\bf e}_{i}^h + \hat {\bf e}_{j}^h,$
    then
    \begin{equation*}
        \psi_{ij}^h - \hat  \psi_{ij}^h = \frac{\|{\bf s}_{B_2}\|^2 - \|\hat {\bf s}_{ij}\|^2}{2}.
    \end{equation*}
    As such, each term of the sum \eqref{E_CountAlgBterm} corresponds to the sub-tournament 
    $\cT_{ij}\subseteq \cT$
    induced by $\{i,j\}$. 
    An analogous relation occurs in type $C_n$ below.
\end{remark}

\subsubsection{Counting in $C_n$}
The argument in type $C_n$ is essentially the same as in type $B_n$, 
except that there are solitaire loop games, and ${\bf e}_i^\ell=2{\bf e}_i$. 
It can be seen that, in this case, 
in addition to the sum \eqref{E_sumalphas}, 
there is a term 
\[
\sum_{i>j} \frac{\psi_{ij}^\ell-\hat \psi_{ij}^\ell}{4}
=2([\Theta_1]+[\Theta_2])
\]
equal to twice the number of neutral clovers in $\cT$. 
The equality follows from analogous argument as in $B_n$, see \cref{F_countC}. 
Therefore, in type $C_n$, 
\begin{equation}\label{E_Ccount}
\frac{\| {\bf s}_\Phi\|^2-\| {\bf s}\|^2}{2}
=[\Delta_c]+[\Delta_b]+2([\Theta_1]+[\Theta_2]),
\end{equation}
as claimed. 

\begin{figure}[h!]
\centering
\includegraphics[scale = 1]{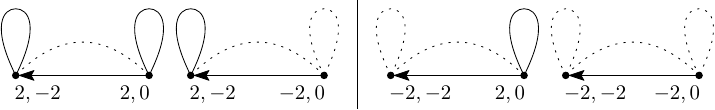}
\caption{Up to symmetries (including reversal), there are four $C_2$-tournaments. 
The tournaments on the left contain a neutral clover and 
the tournaments on the right do not. The net wins by each vertex from its solitaire
and competitive/collaborative games are given.}
\label{F_countC}
\end{figure}

\subsection{Counting combinatorially}
\label{S_embed}

In this section, we sketch an alternative proof of \cref{T_MainCount}. 
Put 
\[
\vartheta(\Phi)=\|{\bf s}_\Phi\|^2/2.
\]
The basis for the arguments in this section are the equalities 
\begin{align*}
\vartheta(A_{n-1})&=n(n-1)(n+1)/24,\\
\vartheta(B_n)&=n(2n-1)(2n+1)/24,\\
\vartheta(C_n)&=n(n+1)(2n+1)/12,\\
\vartheta(D_n)&=n(n-1)(2n-1)/12, 
\end{align*}
which imply
\begin{align}
\label{E_BtoA}
2\vartheta(B_n) &= \vartheta(A_{2n-1}),\\ 
\label{E_CtoA}
2 \vartheta(C_n) &= \vartheta(A_{2n}),\\
\label{E_DtoA}
2\vartheta(D_n) &= \vartheta(A_{2n}) - n^2. 
\end{align} 
As suggested by these identities, 
a Coxeter tournament $\cT$
can be embedded as a specific graph tournament $T=T(\cT)$.

\subsubsection{From $B_n$ to $A_{2n-1}$}
\label{S_BtoA}

Suppose that $\cT$ is a tournament of type $B_n$ with score sequence
${\bf s}(\cT)=(s_1,\ldots,s_n)$. We construct a tournament $T$ of type $A_{2n-1}$ 
with score sequence 
\begin{equation}\label{E_score_BtoA}
{\bf s}(T)=(s_1,\ldots,s_n,-s_1,\ldots,-s_n).
\end{equation}
By \eqref{E_triangles} and \eqref{E_BtoA},
the number of copies of $\Delta_c$ in $T$ satisfies 
\[
[\Delta_c]_{T}
=2\vartheta(B_n)-\sum_{i=1}^n s_i^2. 
\]
Therefore \eqref{E_Bcount} will follow, if we can show that
\begin{equation}\label{E_BtoA2}
[\Delta_c]_\cT + [\Delta_b]_\cT + [\Omega_1]_\cT + [\Omega_2]_\cT + [\Omega_3]_\cT = \frac{1}{2} [\Delta_c]_{T}.
\end{equation}
That is, we need to show that $T$ has twice as many copies of 
$\Delta_c$  
as the total number of generators in $\cT$. 

To describe the construction, we label the $2n$ players in $T$
using $[\pm n]=\{\pm1,\ldots,\pm n\}$. 
\begin{itemize}[nosep]
\item If $i$ wins/loses its (half edge) solitaire game in $\cT$, then $i$ wins/loses
its (competitive) game against $-i$ in $T$. 
\item If $i$ wins/loses its competitive game against $j$ in $\cT$, then in $T$\\
(1) $i$ wins/loses its game against $j$ and\\
(2) $-i$ loses/wins its game against $-j$.
\item If $i$ and $j$ win/lose their collaborative game in $\cT$, then in $T$\\
(1) $i$ wins/loses its game against $-j$ and \\
(2) $j$ wins/loses its game against $-i$. 
\end{itemize}
Note that $T$ is {\it antisymmetric,}
in the sense that the 
reversal of $T$ is the same as $T$ viewed upside down.
Also note that, by construction, players $\pm i$ earn a score of $\pm s_i$ in $T$, 
as in \eqref{E_score_BtoA} above. 
See \cref{F_BtoAemb}. 

\begin{figure}[h!]
\centering
\includegraphics[scale=0.9]{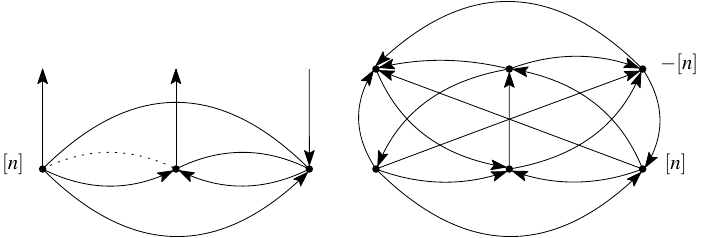}
\caption{A tournament $\cT$ of type $B_n$, 
and its 
antisymmetric
embedding as a tournament $T$ of type $A_{2n-1}$. 
Players with positive/negative labels are in the bottom/top. 
}
\label{F_BtoAemb}
\end{figure}

Next, we sketch the reasons for  \eqref{E_BtoA2}. 
First, observe that for any $\{u,v,w\}\subseteq [\pm n]$, if there is a cyclic  triangle on 
$\{u,v,w\}$ in $T$ then there is also an antisymmetric cyclic triangle on 
$-\{u,v,w\}$ with reversed orientation. 
As such, we may split all cyclic triangles in $T$ into antisymmetric pairs, 
accounting for the $1/2$ in   \eqref{E_BtoA2}. 

We claim that the generators in $\cT$ are in bijective correspondence 
with pairs of antisymmetric cyclic triangles in $T$, in the following way. 
\begin{enumerate}[nosep]
\item 
There is a cyclic triangle in $\cT$ on $\{i,j,k\}$ 
if and only if both triangles on 
$\{i,j,k\}$ and $-\{i,j,k\}$ are cyclically directed in $T$.
\item 
There is a balanced triangle in $\cT$ on vertices $\{i,j,k\}$ involving collaborative games 
between $i,j$ and $j,k$ if and only if 
both triangles on $\{i,-j,k\}$ and $\{-i,j,-k\}$ are cyclically directed in $T$. 
\item 
There is a neutral pair in $\cT$ on $\{i,j\}$ if and only if 
there are two cyclic triangles in the subgraph of $T$ induced on
$\{\pm i,\pm j\}$.
\end{enumerate}

The details of the bijection are omitted. 
Instead, see \cref{F_BtoAgen1,F_BtoAgen2} for the overall idea.

\begin{figure}[h!]
\centering
\includegraphics[scale=1]{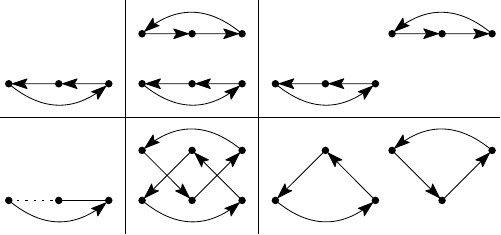}

\caption{Illustration of cases (1) and (2) in the bijective correspondence between
generators in $\cT$ and antisymmetric pairs of cyclic triangles in 
the embedding $T$. The first/second row is the 
case of a cyclic/balanced triangle in $\cT$. 
}
\label{F_BtoAgen1}
\end{figure}

\begin{figure}[h!]
\centering
\includegraphics[scale=1]{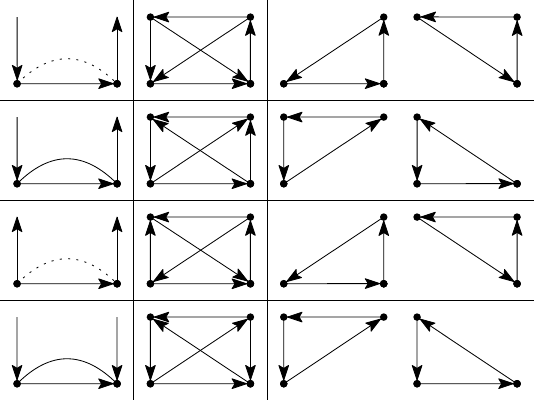}
\caption{Illustration of case (3) in the bijective correspondence. 
Each row is one of the four ways that $\cT$ can 
contain a neutral pair $\Omega_i$. 
}
\label{F_BtoAgen2}
\end{figure}

\subsubsection{From $D_n$ to $A_{2n}$}
\label{S_DtoA}

In a natural way, we can adapt the embedding from $B_n$ to $A_{2n-1}$, 
constructed in 
\cref{S_BtoA}, in order to obtain an embedding of a type $D_n$ 
tournament $\cT$
as a type $A_{2n}$ tournament $T$.
Since type $D_n$ has no solitaire games, we add an extra player $2n+1$. 
For each $i\in[n]$ we put a cyclic triangle on $\{\pm i,2n+1\}$, directed from 
$i$ to $-i$. Note that these
games have net 0 effect on the scores of players $\pm i$. Otherwise, 
the embedding acts on competitive and collaborative games
in the same way as in \cref{S_BtoA}. 
See \cref{F_DtoAemb}. 

\begin{figure}[h!]
\centering
\includegraphics[scale=0.9]{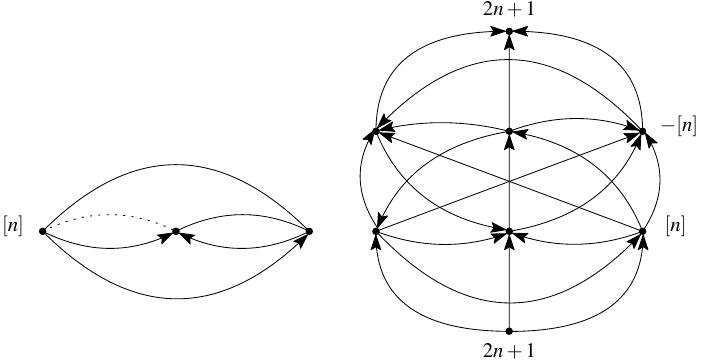}
\caption{A tournament $\cT$ of type $D_n$, 
and its embedding as a tournament $T$ of type $A_{2n}$. 
At right, the top and bottom vertices 
labelled $2n+1$
are the same, 
but drawn in this way for visual clarity. 
}
\label{F_DtoAemb}
\end{figure}

As in \cref{S_BtoA}, the construction gives
\[
{\bf s}(T)=(s_1,\ldots,s_n,-s_1,\ldots,-s_n,0). 
\]
Also, as before, 
every neutral triangle in $\cT$ corresponds to a pair of antisymmetric neutral triangles 
in $T$. In addition to these, there are also $n^2$ {\it layer crossing} cyclic triangles. 
Indeed, there are $n$ cyclic triangles of the form 
$\{\pm i,2n+1\}$, and for $i \neq j \in [n]$ 
there are exactly two cyclic triangles of one of the three forms: 
either 
\begin{itemize}
\item $\{i, -j, 2n+1\}$ and $\{-i, j, 2n+1\}$, or
\item $\{\pm i,j\}$ and $\{\pm i, -j\}$, or 
\item $\{i,\pm j\}$ and $\{-i, \pm j\}$.
\end{itemize}

Therefore,
\[
[\Delta_c]_\cT + [\Delta_b]_\cT 
= \frac{1}{2}\left([\Delta_c]_{T} - n^2\right).
\]
Then, applying \eqref{E_triangles} and \eqref{E_DtoA}, 
it follows that 
\[
[\Delta_c]_\cT + [\Delta_b]_\cT
=\vartheta(D_n)-\frac{1}{2}\sum_{i=1}^n s_i^2,
\]
yielding \eqref{E_Dcount}.

\subsubsection{From $C_n$ to $A_{2n}$}
\label{S_CtoA}

Finally, we note that the embeddings in \cref{S_BtoA,S_DtoA}
can also be modified to obtain a direct embedding from 
$C_n$ into $A_{2n}$. The only difference is that if $i\in[n]$ wins/loses its
loop solitaire game in $\cT$ then $i$ 
wins/loses against $-i$ and $2n+1$, and $2n+1$ wins/loses against $-i$. 
Then, as in \cref{S_DtoA}, 
\[
{\bf s}(T)=(s_1,\ldots,s_n,-s_1,\ldots,-s_n,0). 
\]
See \cref{F_CtoAemb}. 

\begin{figure}[h!]
\centering
\includegraphics[scale=0.9]{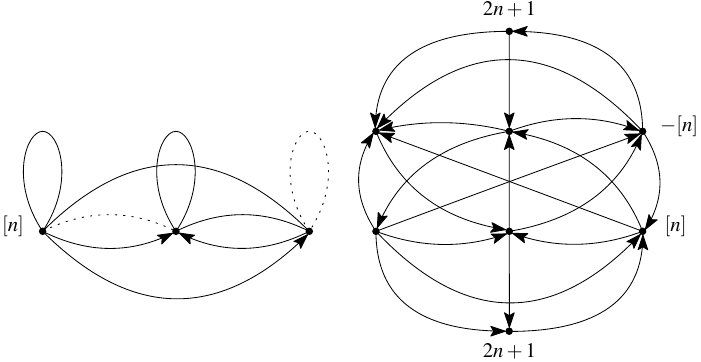}
\caption{A tournament $\cT$ of type $C_n$, 
and its embedding as a tournament $T$ of type $A_{2n}$. 
}
\label{F_CtoAemb}
\end{figure}

Then \eqref{E_Ccount} follows, noting that there is a 
one-to-two correspondence between generators in $\cT$ and generators in $T$.

\section{Acknowledgements}
We thank 
Federico Ardila, 
Christina Goldschmidt, 
James Martin and 
Oliver Riordan for helpful conversations. 
BK was supported by a 
Florence Nightingale Bicentennial Fellowship (Oxford Statistics)
and a Senior Demyship (Magdalen College). 
This publication is based on work partially supported 
(RM) by the EPSRC Centre for Doctoral Training in 
Mathematics of Random Systems: 
Analysis, Modelling and Simulation (EP/S023925/1). 
TP is supported by the Additional Funding Programme for Mathematical Sciences, 
delivered by EPSRC (EP/V521917/1) 
and the Heilbronn Institute for Mathematical Research.

\makeatletter
\renewcommand\@biblabel[1]{#1.}
\makeatother

\providecommand{\bysame}{\leavevmode\hbox to3em{\hrulefill}\thinspace}
\providecommand{\MR}{\relax\ifhmode\unskip\space\fi MR }
\providecommand{\MRhref}[2]{%
  \href{http://www.ams.org/mathscinet-getitem?mr=#1}{#2}
}
\providecommand{\href}[2]{#2}

\end{document}